\documentclass[12pt,a4paper]{amsart}
\usepackage{amsfonts}
\usepackage{amsthm}
\usepackage{amsmath}
\usepackage{amscd}
\usepackage[latin2]{inputenc}
\usepackage{t1enc}
\usepackage[mathscr]{eucal}
\usepackage{indentfirst}
\usepackage{graphicx}
\usepackage{graphics}
\usepackage{pict2e}
\usepackage{epic}
\numberwithin{equation}{section}
\usepackage[margin=2.9cm]{geometry}
\usepackage{epstopdf} 
\usepackage{tikz}
\usepackage{tikz-cd}
\usepackage{amssymb}

\newtheorem{theorem}{Theorem}
\newtheorem{lemma}[theorem]{Lemma}

\newtheorem{proposition}[theorem]{Proposition}

\theoremstyle{definition}

\theoremstyle{remark}

\theoremstyle{remark}

\begin{document}

\title[A Geometric Interpretation of the Normal Closure of $B_n$ in $B_n(T)$]{A Geometric Interpretation of the Normal Closure of the Braid Group $B_n$ in the braid group of the torus $B_n(T)$}

\author{Liming Pang}

\address{New York University \\ Courant Institute of Mathematics, Department of Mathematics \\ 251 Mercer Street\\
New York, N.Y. 10012 \\
U.S.A. } 

\email{liming@cims.nyu.edu}

\begin{abstract} It had been proved in \cite{on braid groups} and \cite{goldberg} that the normal closure of the pure braid group $P_n(D)$ in the pure braid group of the torus $P_n(T)$ is the commutator subgroup $[P_n(T),P_n(T)]$. In this paper we are going to study the case of full braid groups: i.e., the normal closure of $B_n(D)$ in $B_n(T)$, which turns out to have an interesting geometric description.
\end{abstract}

\maketitle

\section{Introduction} 

Let $D$ denote a $2$-dimensional disk, $T$ denote a $2$-dimensional torus and $S$ denote a $2$-dimensional surface. Let $C_n(S)$ denote the unordered configuration space of $n$ points on $S$, and $C'_n(S)$ is the ordered configuration space of $n$ points on $S$. $B_n(S)=\pi_1C_n(S)$ is the full braid group of $n$ strings on $S$ and $P_n(S)=\pi_1C'_n(S)$ is the pure braid group of $n$ strings on $S$. For the detailed definitions of the above concepts, the reader may refer to Section \ref{definition}.

In \cite{on braid groups} it was shown that the abelianization of the pure braid group of $P_n(T)$, $P_n(T)/[P_n(T),P_n(T)]$, is isomorphic to the product of $n$ copies of $\pi_1(T)$, where $[P_n(T),P_n(T)]$ is the commutator subgroup of $P_n(T)$. More specifically, given a surface $S$, consider the inclusion of the ordered configuration space of $S$, $C_n'(S)$ into the direct product of $n$ copies of $S$, $i: C_n'(S)\hookrightarrow S\times ...\times S$. It induces a homomorphism on fundamental groups:
\[
i_*: P_n(S)=\pi_1C_n'(S)\longrightarrow \pi_1(S\times...\times S)\cong \pi_1(S)\times...\times \pi_1(S).
\]
\cite{on braid groups} tells us that when $S$ is the torus $T$, $\ker(i_*)=[P_n(T),P_n(T)]$, the commutator subgroup of $P_n(T)$. 

Later in \cite{goldberg}, it was proved that $\ker{i_*}=P_n(D)^{P_n(T)}$, the normal closure of $P_n(D)$ in $P_n(T)$, i.e., $\ker(i_*)$ is the smallest normal subgroup in $P_n(T)$ containing $P_n(D)$. Combining the above results, we see: 
\[
P_n(D)^{P_n(T)}=[P_n(T),P_n(T)].
\]

The main theorem of this paper is to generalize the result above to the case of full braid groups, i.e., to study the normal closure of $B_n(D)$ in $B_n(T)$.

In Section \ref{core}, we will define a subset of $C_n(T)$, denoted by $C_n^e(T)$, and the goal of this paper is to show the following:

\begin{theorem} ($\textbf{Main Theorem}$)
$\pi_1C_n^e (T)$ is the normal closure of $B_n(D)$ in $B_n(T)$. 
\end{theorem}

\section{Acknowledgment}

This result first appears in the author's doctoral dissertation. 

I would like to thank Prof. Sylvain Cappell and Prof. Edward Miller for their valuable discussions and suggestions. 

\section{Definitions and Preliminary}\label{definition}

The geometric braid group $B_n$ is defined as follows: $\{x_1,...,x_n\}$ is a set of $n$ interior points in the closed unit disk $D$ in the plane $\mathbb{R}^2$. Consider the set of $n$ disjoint continuous paths $\{\alpha_1(t),...,\alpha_n(t)\}$ in $D\times [0,1]$ with $\alpha_i(0)=(x_i,0)$, $\{\alpha_1(1),...,\alpha_n(1)\}=\{x_1,...,x_n\}\times \{1\}$, $\alpha_i(t)$ is in the interior of the disk $D\times\{t\}$ for any $t\in [0,1]$. We quotient this set by homotopies of the strings relative to their endpoints. The quotient set can be made into a group by defining the operation to be concatenation of two braids, i.e., identify $D\times \{1\}$ of the first braid with $D \times \{0\}$ of the second braid to obtain a new braid. The product $D\times I$ is embedded in $\mathbb{R}^2\times \mathbb{R}=\mathbb{R}^3$, and the $n$ strings can then be visualized in this manner.\\

$B_n$ has the classical Artin's presentation \cite{artin}:
\[
<\sigma_1,...,\sigma_{n-1}| \sigma_1\sigma_j=\sigma_j\sigma_i (|i-j|>1),\sigma_i\sigma_j\sigma_i=\sigma_j\sigma_i\sigma_j(|i-j|=1)>,
\]
and geometrically $\sigma_i$ corresponds to a counterclockwise half twist of the $i$-th and $(i+1)$-th strings.\\

The concept of configuration spaces, introduced by Fadell and Neuwirth in \cite{conf space}, brings a topological description of the braid group as a fundamental group, and a possible generalization of the concept of braid group to any surface. \\  

Let $S$ be a 2 dimensional surface. The $\textbf{ordered configuration space}$ of $n$ points on $S$ is defined to be $C'_n(S)=\{(x_1,...,x_n)\in S\times ...\times S| x_i \neq x_j \text{ for } i\neq j\}$. The symmetric group $S_n$ acts on $C'_n(S)$ by permuting the $n$ coordinates. The orbit space of this action is the $\textbf{unordered configuration space}$ of $n$ points on $S$: $C_n(S)=C'_n(S)/{S_n}=\{\{x_1,...,x_n\}\in 2^S| x_i \neq x_j \text{ for } i\neq j\}$. In other words, $C_n(S)$ consists of subsets of $S$ of cardinality $n$. 

There is a natural covering space structure corresponding to the above action:
\[
C_n'(S) \longrightarrow C_n(S)
\]
\[
(x_1,...,x_n)\mapsto \{x_1,...,x_n\}.
\]
The covering map is to forget the ordering of the elements. This is a regular covering with the covering transformation to be the symmetric group $S_n$. \\

After choosing a base point $x^0=\{x_1^0,...,x_n^0\}\in C_n(S)$, we can define the $\textbf{braid group}$ of n strands on $S$, $B_n(S)$, to be the fundamental group of $C_n(S)$: $B_n(S)=\pi_1(C_n(S,),x^0)$. Since different choices of base points will lead to isomorphic fundamental groups, we will omit the base point in the following discussions, unless it is necessary to stress on the base point. Similarly, the $\textbf{pure braid group}$ of $n$ strands on $S$, $P_n(S)$, is defined to be the fundamental group of $C'_n(S)$: $P_n(S)=\pi_1(C'_n(S))$. When $S=D$, a disk, we get the classical braid groups $B_n=B_n(D)=\pi_1(C_n(D))$ and $P_n=P_n(D)=\pi_1(C'_n(D))$.

The covering space mentioned above leads to the following short exact sequence:

\[
1 \longrightarrow P_n(S) \longrightarrow B_n(S) \longrightarrow S_n \longrightarrow 1.
\] 

If $f:S_1\longrightarrow S_2$ is an injective continuous map between two surfaces, there is an induced map on their configuration spaces, which we denote by $\hat{f}$: 

\[
\hat{f}: C_n(S_1) \longrightarrow C_n(S_2)
\]
\[
\{x_1,...,x_n\} \mapsto \{f(x_1),...,f(x_n)\}.
\]
The maps on configurations spaces induce homomorphisms on the braid groups, by the fundamental group construction: 
\[
\hat{f}_*: B_n(S_1) \longrightarrow B_n(S_2).
\]
 
The following Proposition proved in \cite{rolfsen} gives a criterion for when $f_*$ is injective: 

\begin{proposition} (Paris and Rolfsen)
Let $f:S_1\hookrightarrow S_2$ be an inclusion map. Let $S_1$ be different from the sphere or the projective plane, and let $S_2$ be such that none of the connected components of $\overline{S_2\setminus S_1}$ is a disk, then the homomorphism $\hat{f}_*: B_n(S_1)\longrightarrow B_n(S_2)$ is injective.
\end{proposition}

As remarked in \cite{rolfsen}, a special case of this proposition when $S_1$ is a disk is proved in \cite{goldberg}.

By this Proposition, we can see that $B_n(D)$ injects into $B_n(T)$ canonically.

\section{The Subspace $C_n^e(T)$ And Its Properties}\label{core}

The torus $T\cong S^1\times S^1$ is a topological abelian group, and we denote the identity element by $e\in T$. 

Define a subspace of $C_n(T)$ by:
\[
C^e_n(T)=\{\{x_1,...x_n\}\in C_n(T): \prod_{i=1}^{n}x_i=e\},
\]
that is, an element in this subspace is a set of $n$ distinct points on torus whose product is $e$. Since the torus is abelian, this subspace is well-defined. 

Similarly, for the ordered case, we can also define a subspace of $C'_n(T)$ by:
\[
C'^e_n(T)=\{(x_1,...x_n)\in C_n(T): \prod_{i=1}^{n}x_i=e\}.
\]

The groups $\pi_1(C^e_n(T))$ contains some groups that we are familiar with as subgroups:

\begin{lemma}\label{subgroup}
$B_n(D)$ is a subgroup of $\pi_1(C^e_n(T))$, and $P_n(D)$ is a subgroup of $\pi_1({C'}_n^e(T))$. 
\end{lemma}

\begin{proof}
We choose the base point $\{x_1^0,...,x_n^0\}$ of $C_n(T)$ to be inside $C_n^e(T)$, so $\prod_{i=1}^n x_i^0=e$. We only need to illustrate that the generators of $B_n(D)$, $\sigma_i$, are in $\pi_1(C^e_n(T))$. We can choose a representative of $\sigma_i\in B_n(T)$ by letting the $i$-th and $(i+1)$-th base points do a half twist in a way that at each time they are moving in the opposite way, which makes that at any time, the product of the $n$ points will remain to be $e$, so we obtain a loop in $C_n^e(T)$, hence it represents an element in $\pi_1(C^e_n(T))$. We conclude $B_n(D)\subseteq \pi_1(C^e_n(T))$. 

$\pi_1({C'}_n^e(T))=\pi_1(C^e_n(T))\cap P_n(T)$, so $P_n(D)=B_n(D)\cap P_n(T)\subseteq \pi_1(C^e_n(T)) \cap P_n(T)=\pi_1({C'}_n^e(T))$.
\end{proof}

The next lemma shows we actually get a fiber bundle from our construction: 

\begin{lemma}\label{bundle}
The map
\[
m:C_n(T) \longrightarrow T
\]
\[
\{x_1,...,x_n\}\mapsto \prod_{i=1}^nx_i
\]
gives a fiber bundle with fiber $C_n^e(T)$.

So we have a short exact sequence:
\[
1 \longrightarrow \pi_1(C^e_n(T)) \longrightarrow B_n(T) \xrightarrow{m_*} \pi_1(T) \longrightarrow 1.
\]
\end{lemma}

\begin{proof}
It suffices to show for a small neighborhood $U$ of $e$ it holds $U\times C_n^e(T)\cong m^{-1}(U)$. Since $T=S^1\times S^1$, if we take $n$-th roots of unity, there are $n^2$ of them. If we write the torus as $\mathbb{R}^2/\mathbb{Z}^2$, then $e=[(0,0)]$. Let $U=\{[(x,y)]\in \mathbb{R}^2/\mathbb{Z}^2: |x|<\frac{1}{2}, |y|<\frac{1}{2} \}$, we can define the homeomorphism:
\[
U\times C^e_n(T)\longrightarrow m^{-1}(U)
\]
\[
([(x,y)], \{[(x_1,y_1)],...,[(x_n,y_n)]\})\mapsto \{[(x_1+\frac{x}{n},y_1+\frac{y}{n})],...,[(x_n+\frac{x}{n},y_n+\frac{y}{n})]\}.
\]
In the above formula we choose $(x,y)$ representing $[(x,y)]$ such that $|x|<\frac{1}{2}$, $|y|<\frac{1}{2}$

The short exact sequence follows from this fibration.

\end{proof}

There is another combinatorial description of the map $m_*$, making use of the group representation of $B_n(T)$ in \cite{presentation}: 

Generators: 
\[
a,b, \sigma_1,...,\sigma_{n-1}
\]

Relations:

\begin{align*}
\sigma_i\sigma_j &=\sigma_j\sigma_i, (1\leq i\neq j\leq n-1)\\
\sigma_i\sigma_{i+1}\sigma_i &=\sigma_{i+1}\sigma_i\sigma_{i+1}, (1\leq i \leq n-1)\\
a\sigma_i&=\sigma_ia,(2\leq i\leq n-1)\\
b\sigma_i&=\sigma_ib,(2\leq i\leq n-1)\\
\sigma_1^{-1}a\sigma_1^{-1}a &=a\sigma_1^{-1}a\sigma_1^{-1}\\
\sigma_1^{-1}b\sigma_1^{-1}b &=b\sigma_1^{-1}b\sigma_1^{-1}\\
\sigma_1^{-1}a\sigma_1^{-1}b &=b\sigma_1^{-1}a\sigma_1\\
ab^{-1}a^{-1}b &=(\sigma_1\sigma_2...\sigma_{n-1})(\sigma_{n-1}...\sigma_2\sigma_1)
\end{align*}

Define the group homomorphism $\phi: B_n(T)\longrightarrow \mathbb{Z}\times\mathbb{Z}$ by $\phi(a)=(1,0)$, $\phi(b)=(0,1)$ and $\phi(\sigma_i)=(0,0)$. It is not hart to see this is a well-defined homomorphism by studying the relations above. 

If we identify $H_1(T;\mathbb{Z})=\mathbb{Z}\times\mathbb{Z}$ by $(1,0)$, $(0,1)$ representing longitude and meridean $1$-cycles respectively, we then can interpret $m_*$ in the following way:

\begin{lemma}
The map $m_*$ is the same as the map $\phi$. 
\end{lemma}

\begin{proof}
By the definition of the map $m$, it is clear that $m_*(a)=(1,0),m_*(b)=(0,1)$ (or $m_*(a)=(0,1),m_*(b)=(1,0)$) and $m_*(\sigma_i)=(0,0)$.
\end{proof}

A remark is that it follows easily from the group presentation given in \cite{presentation} that the quotient of $B_n(S)$ by $B_n(D)^{B_n(S)}$ is isomorphic to $H_1(S;\mathbb{Z})$ for a closed orientable surface $S$ with positive genus.

\begin{lemma} 

The following diagram commutes and each of the horizontal lines is a short exact sequence, where $(P_n(T))'=[P_n(T),P_n(T)]$ and $H=(P_n(T))'B_n(D)=\{xy\in B_n(T)| x\in (P_n(T))', y\in B_n(D)\}$. The first two columns of vertical lines are all inclusion maps: 

\[
\begin{tikzcd}
1\arrow{r} & P_n(D)\arrow[hook]{r}\arrow[hook]{d} & B_n(D)\arrow{r}\arrow[hook]{d} & S_n\arrow{r}\arrow{d}{=} & 1\\
1\arrow{r} & (P_n(T))'\arrow[hook]{r}\arrow[hook]{d} & H\arrow{r}\arrow[hook]{d} & S_n\arrow{r}\arrow{d}{=} & 1\\
1\arrow{r} & \pi_1({C'}_n^e(T))\arrow[hook]{r}\arrow[hook]{d} & \pi_1 (C_n^e(T))\arrow{r}\arrow[hook]{d} & S_n\arrow{r}\arrow{d}{=} & 1\\
1\arrow{r} & P_n(T)\arrow[hook]{r} & B_n(T)\arrow{r} & S_n\arrow{r} & 1\\
\end{tikzcd}
\]
\end{lemma}

\begin{proof}
We first verify all the vertical maps are inclusions. We have shown that $P_n(D)\subseteq \pi_1({C'}_n^e(T))$, and by \cite{goldberg}, $(P_n(T))'$ is the normal closure of $P_n(D)$ in $P_n(T)$, so $P_n(T)'\subseteq \pi_1({C'}_n^e(T))$. 

For the other column, we need to verify $H=(P_n(T))'B_n(D)\subseteq \pi_1(C')^e_n(T)$. This follows easily from $(P_n(T))'\subseteq \pi_1({C'}_n^e(T)) \subseteq \pi_1 (C^e_n(T))$ and $B_n(D)\subseteq \pi_1 (C^e_n(T))$. 

The short exact sequences are clear, except for the second row. The image of $H=(P_n(T))'B_n(D)\longrightarrow S_n$ is determined by the $B_n(D)$ part, since $(P_n(T))'$ are pure braids. So the kernel of this map should be $(P_n(T))'P_n(D)=(P_n(T))'$, because $P_n(D)\subseteq (P_n(T))'$. 

That the diagram is commutative is because the vertical lines and the second horizontal map on each row are all inclusions. 
\end{proof}

Similar to the notation $(P_n(T))'=[P_n(T),P_n(T)]$, we will write $(B_n(T))'=[B_n(T),B_n(T)]$.

\begin{lemma}
$B_n(T)/(B_n(T))'\cong \mathbb{Z}\times \mathbb{Z}\times \mathbb{Z}_2$.
\end{lemma}

\begin{proof}
By adding the relations $ab=ba$, $a\sigma_i=\sigma_ia$, $b\sigma_i=\sigma_ib$ ($1\leq i\leq n-1$) and $\sigma_i\sigma_j=\sigma_j\sigma_i$ ($1\leq i< j\leq n-1$) to the presentation of $B_n(T)$ in \cite{presentation} and doing reductions on the relations, we obtain a group presentation for $B_n(T)/(B_n(T))'$:
\[
<a,b,\sigma|ab=ba,a\sigma=\sigma a,b\sigma=\sigma b,\sigma^2=1>
\]
which coincides with the group presentation for $\mathbb{Z}\times \mathbb{Z}\times \mathbb{Z}_2$.
\end{proof}

\begin{lemma}\label{commutator}
$\pi_1 (C_n^e(T))=(B_n(T))'B_n(D)$.
\end{lemma}

\begin{proof}
The following diagram with two short exact sequences commutes:
\[
\begin{tikzcd}
{}& & 1\arrow{d} & & \\
  & & \pi_1(C_n^e(T))\arrow{d} & &\\
 1\arrow{r} & (B_n(T))'\arrow{r} & B_n(T)\arrow{r}\arrow{d}{m_*} & \mathbb{Z}\times \mathbb{Z}\times\mathbb{Z}_2\arrow{r}\arrow{dl} & 1 \\
  & & \mathbb{Z}\times\mathbb{Z}\arrow{d} & & \\
  & &  1 & & \\
\end{tikzcd}
\]

The horizontal short exact sequence describes the abelianization of $B_n(T)$. 

So we see that $(B_n(T))'$ is a subgroup of $\pi_1 (C_n^e(T))$, and also $B_n(D)$ is a subgroup of $\pi_1 (C_n^e(T))$, we get $(B_n(T))'B_n(D)\subseteq \pi_1 (C_n^e(T))$.

Apply the Third Isomorphism Theorem to $(B_n(T))' \trianglelefteq \pi_1(C_n^e(T)) \trianglelefteq B_n(T)$, we see:
\[
\pi_1(C_n^e(T))/(B_n(T))' \cong \mathbb{Z}_2
\]
Since $\sigma_1\in \pi_1(C_n^e(T))$ and $\sigma_1\notin (B_n(T))'$, $\sigma_1$ will represent the nonzero coset in the above quotient, so:
\[
(B_n(T))'B_n(D)\subseteq \pi_1(C_n^e(T))= (B_n(T))' \sqcup (B_n(T))'\sigma_1 \subseteq (B_n(T))'B_n(D).
\] 
\end{proof}

\section{Proof of the Main Theorem}

\begin{proof}
First, it is easy to see $\pi_1(C^e_n(T))\supseteq B_n(D)^{B_n(T)}$ since $\pi_1(C^e_n(T))=\ker(m_*)$, and $B_n(D)\subseteq \ker(m_*)$, so $\pi_1(C^e_n(T))$ is a normal subgroup of $B_n(T)$ containing $B_n(D)$, it follows $\pi_1(C^e_n(T))\supseteq B_n(D)^{B_n(T)}$.

On the other hand, we need to show $\pi_1(C^e_n(T))\subseteq B_n(D)^{B_n(T)}$. Since by Lemma \ref{commutator}, $\pi_1(C^e_n(T))= (B_n(T))'B_n(D)$, it suffices to show $(B_n(T))'\subseteq B_n(D)^{B_n(T)}$.

Consider the image of an element of form $\alpha\beta \alpha^{-1}\beta^{-1}\in (B_n(T))'$, where $\alpha,\beta\in B_n(T)$ under the quotient map $\pi_1(C^e_n(T))\xrightarrow{q} \pi_1(C^e_n(T))/B_n(D)^{B_n(T)}$. Since $B_n(T)=P_n(T)\cdot B_n(D)$, we can write $\alpha=p_1t_1$, $\beta=p_2t_2$, where $p_1,p_2\in P_n(T)$ and $t_1,t_2\in B_n(D)$. So $\alpha\beta\alpha^{-1}\beta^{-1}=p_1t_1p_2t_2t_1^{-1}p_1^{-1}t_2^{-1}p_2^{-1}$, and the fact $B_n\subseteq B_n(D)^{B_n(T)}$ implies $q(\alpha\beta\alpha^{-1}\beta^{-1})=q(p_1p_2p_1^{-1}p_2^{-1})\in \pi_1(C^e_n(T))/B_n(D)^{B_n(T)}$. 

But by \cite{goldberg}, $p_1p_2p_1^{-1}p_2^{-1}\in (P_n(T))'=P_n(D)^{P_n(T)}\subseteq B_n(D)^{B_n(T)}$, it follows $q(\alpha\beta\alpha^{-1}\beta^{-1})$ is the identity element in the quotient group $\pi_1(C^e_n(T))/B_n(D)^{B_n(T)}$, which means $\alpha\beta \alpha^{-1}\beta^{-1}\in B_n(D)^{B_n(T)}$. We can conclude $(B_n(T))'\subseteq B_n(D)^{B_n(T)}$, which implies $\pi_1(C^e_n(T))\subseteq B_n(D)^{B_n(T)}$.

\end{proof}

\end{document}